\documentclass{amsart}

\usepackage{amssymb}
\usepackage{amscd}
\usepackage[all,cmtip]{xy}
\usepackage{hyperref}
\usepackage{graphicx}
\usepackage{cite}

\newcommand{\comment}[1]{}

\newcommand{\Z}{\mathbf{Z}}

\theoremstyle{plain}
\newtheorem{theorem}{Theorem}[section]

\newtheorem{lemma}[theorem]{Lemma}

\newtheorem{proposition}[theorem]{Proposition}

\theoremstyle{definition}
\newtheorem{definition}[theorem]{Definition}

\newtheorem{remark}[theorem]{Remark}

\begin{document}

\title[]{On the last fall degree of zero-dimensional Weil descent systems}
\author[]{Ming-Deh A. Huang (USC, mdhuang@usc.edu), Michiel Kosters (TL@NTU, kosters@gmail.com), Yun Yang (NTU, YANG0379@e.ntu.edu.sg), Sze Ling Yeo (I2R, slyeo@i2r.a-star.edu.sg)}
\address{}
\email{}
\urladdr{}
\date{\today}

\maketitle

\begin{abstract}
In this article we will discuss a new, mostly theoretical, method for solving (zero-dimensional) polynomial systems, which lies in between Gr\"obner basis computations and the heuristic first fall degree assumption and is not based on any heuristic. This method relies on the new concept of \emph{last fall degree}.

Let $k$ be a finite field of cardinality $q^n$ and let $k'$ be its subfield of cardinality $q$. Let $\mathcal{F} \subset k[X_0,\ldots,X_{m-1}]$ be a finite subset generating a zero-dimensional ideal. We give an upper bound of the last fall degree of the Weil descent system of $\mathcal{F}$, which depends on $q$, $m$, the last fall degree of $\mathcal{F}$, the degree of $\mathcal{F}$ and the number of solutions of $\mathcal{F}$, but not on $n$. This shows that such Weil descent systems can be solved efficiently if $n$ grows. In particular, we apply these results for multi-HFE and essentially show that multi-HFE is insecure.

Finally, we discuss that the degree of regularity (or last fall degree) of Weil descent systems coming from summation polynomials to solve the elliptic curve discrete logarithm problem might depend on $n$, since such systems without field equations are not zero-dimensional. 
\end{abstract}

\section{Introduction}
Let $k$ be a field and let $\mathcal{F} \subset R=k[X_0,\ldots,X_{m-1}]$ be a finite subset. Let $R_{\leq i}$ be the set of polynomials in $R$ of degree at most $i$. Suppose that we want to find the solutions of $\mathcal{F}$ in $\overline{k}^m$. 

One of the most common methods is the following. First fix a monomial order on $R$, such as the degree reverse lexicographic order, and then compute a Gr\"obner basis of the ideal generated by $\mathcal{F}$ using for example $F_4$ or $F_5$ \cite{FAU4, FAU5}. Then one computes a Gr\"obner basis for the lexicographic order using FGLM \cite{FAU6}. It is often very hard to estimate the complexity of such algorithms. The largest degree which one sees in such a computation of a Gr\"obner basis for the degree reverse lexicographic order is called the \emph{degree of regularity}, and this degree essentially determines the complexity of such algorithms.

One approach to obtain heuristic complexity bounds is the use of the so-called \emph{first fall degree assumption}. For $i \in \Z_{\geq 0}$, we let $V_{\mathcal{F},i}$ be the smallest $k$-vector space such that
\begin{enumerate}
\item
$\{f \in \mathcal{F}: \deg(f) \leq i \} \subseteq V_{\mathcal{F},i}$;
\item
if $g \in V_{\mathcal{F},i}$ and if $h \in R$ with $\deg(hg) \leq i$, then $hg \in V_{\mathcal{F},i}$.
\end{enumerate}
The first fall degree is defined to be the first $d$ such that $V_{\mathcal{F},d} \cap R_{\leq d-1} \neq V_{\mathcal{F},d-1}$ (and if it does not exist, it is defined to be $0$; note that this definition of the first fall degree differs slightly from most definitions as in \cite{PET}, but behaves a lot better). 
The heuristic claim is that the first fall degree is close to the degree of regularity for many systems (see for example \cite{PET}). A quote from \cite{DIN} is ``Our conclusions rely on no heuristic assumptions beyond the standard assumption that the Gr\"obner basis
algorithms terminate at or shortly after the degree of regularity'' (note that in \cite{DIN} the definition of degree of regularity coincides with the first fall degree definition of \cite{PET}). It is quite often easy to give an upper bound on the first fall degree, just by counting arguments (see \cite{DIN} for example). However, in \cite{KO14}, the second and third author of this article raise doubt to the first fall degree heuristic.

In the first part of this article we will try to rectify the situation. We will define the notion of \emph{last fall degree} (or \emph{maximal gap degree}), which is the largest $d$ such that $V_{\mathcal{F},d} \cap R_{\leq d-1} \neq V_{\mathcal{F},d-1}$. We denote the last fall degree of $\mathcal{F}$ by $d_{\mathcal{F}}$. If $\mathcal{F}$ is zero-dimensional with at most $e$ solutions over the algebraic closure of $k$, we show how one can solve the system using $V_{\mathcal{F},\max(d_{\mathcal{F}},e)}$ and monovariate factoring algorithms (Proposition \ref{666}). We will also prove different properties of the last fall degree, for example, that it is always bounded by the degree of regularity and that it behaves well with respect to certain operations (such as linear change of variables and linear change of equations). See Subsection \ref{906} for a comparison with other methods for solving systems, most notably with MutantXL.

In the second part of this article we will show one application of the last fall degree. Basically, if $k$ is a finite field of cardinality $q^n$ and $k'$ is its subfield of cardinality $q$, and $\mathcal{F}$ is zero-dimensional, then we show that the first fall degree of a Weil descent system of $\mathcal{F}$ to $k$ does not depend on $n$. This generalizes practical and mathematical results, if $m=1$ \cite{BET, DIN, FAU3, PET2}. This shows that some versions of multi-HFE (HFE stands for hidden field equations) are much easier to tackle than one would expect. Let us now give a precise formulation of the main theorem.

We denote by $Z(\mathcal{F})$ the set of zeros of $\mathcal{F}$ over $\overline{k}$.
For $r \in \Z_{\geq 0}$ and $c, t \in \Z_{\geq 1}$ we set 
\begin{eqnarray*}
\tau(r,c,t)=\max\left( \lfloor 2 t(c-1) \left(\log_c\left(  \frac{r}{2t}\right)+1 \right) \rfloor,0\right).
\end{eqnarray*}
Note that this function increases when $r$ increases.

\begin{theorem} \label{1234}
Let $k$ be a finite field of cardinality $q^n$. Let $\mathcal{F} \subset R$ be a finite subset. Let $I$ be the ideal generated by $\mathcal{F}$. Assume that the following hold:
\begin{itemize}
\item $I$ is zero-dimensional, say one has $|Z(\mathcal{F})| \leq s$;
\item $I$ is radical;
\item there is a coordinate $t$ such that the projection map $Z(\mathcal{F}) \to \overline{k}$ to coordinate $t$ is injective;
\end{itemize}
Let $\mathcal{F}_f'$ be the Weil descent system of $\mathcal{F}$ to the subfield $k'$ of cardinality $q$ using some basis of $k/k'$, together with the field equations (Subsection \ref{551}). Then one has
\begin{eqnarray*}
d_{\mathcal{F}'_f} \leq \max\left(\tau(\max(d_{\mathcal{F}}, \deg(\mathcal{F}), (m+1)s,1),q,m),m \cdot \tau(2s,q,1), q \right).
\end{eqnarray*}
\end{theorem}

When $m=1$, we obtain a slightly stronger version (Theorem \ref{8056}).

In Section \ref{230} we will explain why Theorem \ref{1234} is not useful to determine the complexity of solving systems coming from summation polynomials for the elliptic curve discrete logarithm problem.

Parts of the results in this article can be found in our paper \cite{HUA}, which will be presented at Crypto 2015. In that paper however, we only restrict to the case when $m=1$ and we leave out certain mathematical proofs.

\subsection{Organization of the paper}
In Section \ref{301} we discuss the last fall degree. We will also discuss how one can solve zero-dimensional systems using the last fall degree and we will compare this method with other methods. We also compare our methods with existing methods.
In Section \ref{302} we introduce Weil descent and an alternative version of Weil descent.
Section \ref{303} is devoted to the proof of Theorem \ref{1234}. In this section we first discuss the relation between the two Weil descent systems. Then we study the monovariate case and deduce the result for the multivariate case from the monovariate case using projection polynomials. Finally, we discuss how one can generalize the main theorem. 
In Section \ref{304} we discuss the relation with multi-HFE. In Section \ref{230} we discuss why the results in this article are not directly useful for studying systems coming from summation polynomials for the elliptic curve discrete logarithm problem.

\section{Last fall degree} \label{301}

In this section we introduce the notion of the \emph{last fall degree} of a system of polynomials. This notion is a parameter for the complexity of solving the polynomial system, and is independent of any monomial order. Later, we will use this notion to study the complexity of Weil descent systems.

Let $k$ be a field and let $R=k[X_0,\ldots,X_{m-1}]$ be a polynomial ring. Note that the affine group $\mathrm{Aff}_m(k) = k^m \rtimes \mathrm{GL}_m(k)$ acts on $R$ by affine change of variables. This action preserves the total degree. The set of polynomials of degree $\leq i$ is denoted by $R_{\leq i}$.

Let $\mathcal{F}$ be a finite subset of $R$ and let $I \subseteq R$ be the ideal generated by $\mathcal{F}$. We set $\deg(\mathcal{F})=\max \{\deg(f): f \in \mathcal{F}\}$. Furthermore, we set $\deg_{X_i}(\mathcal{F})=\max\{ \deg_{X_i}(f): f \in \mathcal{F}\}$. 

\subsection{Constructible polynomials}

\begin{definition}
For $i \in \Z_{\geq 0}$, we let $V_{\mathcal{F},i}$ be the smallest $k$-vector space such that
\begin{enumerate}
\item
$\mathcal{F} \cap R_{\leq i}=\{f \in \mathcal{F}: \deg(f) \leq i \} \subseteq V_{\mathcal{F},i}$;
\item
if $g \in V_{\mathcal{F},i}$ and if $h \in R$ with $\deg(hg) \leq i$, then $hg \in V_{\mathcal{F},i}$.
\end{enumerate}
We set $V_{\mathcal{F}, \infty}=I$. For convenience, we set $V_{\mathcal{F},-1}=\emptyset$. 
\end{definition}

If $\mathcal{F}$ is fixed, we just write $V_i$ instead of $V_{\mathcal{F},i}$.
Intuitively, $V_i$ is the largest subset of $I$ which can be constructed from $\mathcal{F}$ by doing operations of degree at most $i$.
Note that $V_i$ is a finite-dimensional $k$-vector space of dimension 
\begin{eqnarray*}
\mathrm{dim}_{k}(V_i) \leq \dim_k R_{\leq i}={{m+i}\choose{i}} \leq (m+i)^{i}.
\end{eqnarray*}
Notice that for any $f \in I$, there is an $i \in \Z_{\geq 0}$ such that $f \in V_i$. Phrased differently, we have $I=V_{\infty} = \bigcup_{i \in \Z_{\geq 0}} V_i$.

\begin{definition}
For $g,h \in R$ and $i \in \Z_{\geq 0} \sqcup \{\infty\}$. we write $g \equiv_{\mathcal{F},i} h$ if $g-h \in V_{\mathcal{F},i}$. If $\mathcal{F}$ is fixed, we often write $g \equiv_i h$. We write $g \equiv h$ if $g \equiv_{\infty}h$, which means $g-h \in I$.
\end{definition}

\begin{proposition} \label{4213}
Let $\mathcal{F}, \mathcal{G} \subset R$ be finite subsets, $i \in \Z_{\geq 0}$, $A \in \mathrm{Aff}_m(k)$ and $k'/k$ a field extension. Then the following hold:
\begin{enumerate}
\item $V_{\mathcal{F},i}$ can be constructed in a number of field operations which is polynomial in $(m+i)^i$ and in the cardinality of $\mathcal{F}$.
\item if $\mathcal{F} \subseteq \mathcal{G}$, then $V_{\mathcal{F},i} \subseteq V_{\mathcal{G},i}$;
\item if $\mathrm{Span}_k(\mathcal{F}) = \mathrm{Span}_k(\mathcal{G})$ and $i \geq \deg(\mathcal{F})$, then $V_{\mathcal{F},i}=V_{\mathcal{G},i}$;
\item one has $A V_{\mathcal{F},i}= V_{A \mathcal{F},i}$;
\item one has $V_{\mathcal{F},i} \otimes_k k' = V_{\{f \otimes_k 1:\ f \in \mathcal{F} \},i} \subset k'[X_0,\ldots,X_{m-1}]$.
\end{enumerate}
\end{proposition}
\begin{proof}
i: One can construct the $V_{\mathcal{F},i}$ using linear algebra as follows. Use a degree preserving ordered basis of $R_{\leq i}$ and use row echelon forms to construct the $V_{\mathcal{F},i}$.

ii, iii,v: Follows directly from the definitions.
 
iv: Follows because the action of $\mathrm{Aff}_m(k)$ respects degrees.
\end{proof}

\begin{remark} \label{yy}
Let $f_1, f_2, g_1, g_2 \in R$. Assume $f_1 \equiv_i f_2$, $g_1 \equiv_j g_2$. Assume that $\deg(f_1) \leq i$ and $\deg(g_2) \leq j$. Then one has
\begin{eqnarray*}
f_1g_1-f_2g_2 = f_1(g_1-g_2)+g_2(f_1-f_2) \in V_{i+j}.
\end{eqnarray*}
Hence we have $f_1 g_1 \equiv_{i+j} f_2 g_2$.
\end{remark}

\subsection{Last fall degree}

We now define the last fall degree.

\begin{definition}
Let $\mathcal{F}$ be a finite subset of $R$ and let $I$ be the ideal generated by $\mathcal{F}$.
The minimal $d \in \Z_{\geq 0} \sqcup \{\infty\}$ such that for all $f \in I$ we have $f \in V_{\max(d, \deg(f))}$, is called the \emph{last fall degree} of $\mathcal{F}$, and is denoted by $d_{\mathcal{F}}$.
\end{definition}

Note that the above definition implies that for $i \geq d_{\mathcal{F}}$, one has $V_{\mathcal{F},i}=I \cap R_{\leq i}$.

We will now state some of the properties of the last fall degree.

\begin{proposition} \label{4444}
Let $\mathcal{F}, \mathcal{G} \subset R$ be finite subsets which generate ideals $I$ respectively $J$. Let $A \in \mathrm{Aff}_m(k)$ and $k'/k$ be a field extension. The following hold.
\begin{enumerate}
\item One has: $d_{\mathcal{F}} \in \Z_{\geq 0}$.
\item Let $\mathcal{B}$ be a Gr\"obner basis with respect to some degree refining monomial order on $R$. Then there is an integer $c \in \Z_{\geq 0}$ such that $\mathcal{B} \subseteq V_{\mathcal{F},c}$ and one has $d_{\mathcal{F}} \leq c$. 
\item One has: $d_{\mathcal{F}}$ is the largest $c \in \Z_{\geq 0}$ such that $V_c \cap R_{\leq c-1} \neq V_{c-1}$.
\item If $\mathrm{Span}_k(\mathcal{F}) = \mathrm{Span}_k(\mathcal{G})$, then one has $\max(d_{\mathcal{F}}, \deg(\mathcal{F}))= \max(d_{\mathcal{G}}, \deg(\mathcal{F}))$.
\item One has: $d_{\mathcal{F}}=d_{A\mathcal{F}}$.
\item Consider the set $\{f \otimes 1:\ f \in \mathcal{F}\} \subset k'[X_0,\ldots,X_{m-1}]$. One has: $d_{\{f \otimes 1:\ f \in \mathcal{F}\}} = d_{\mathcal{F}}$.
\item If $I=J$ and $\mathcal{F} \subseteq \mathcal{G}$, then one has $d_{\mathcal{G}} \leq d_{\mathcal{F}}$.
\item If $g \in V_{\mathcal{F},j}$, then one has $d_{\mathcal{F}} \leq \max(j, d_{\mathcal{F} \cup \{g\}})$.
\end{enumerate}
\end{proposition}
\begin{proof}
i, ii: i follows from ii directly, since a Gr\"obner basis always exists. It is easy to see that there is a $c$ with $\mathcal{B} \subseteq V_{\mathcal{F},c}$. Take $f \in I$ and write $f=\sum_{b \in \mathcal{B}} a_b b$ with $\deg(a_b b) \leq \deg(f)$ for $b \in \mathcal{B}$. This is possible because $\mathcal{B}$ is a Gr\"obner basis. Then one easily finds $f \in V_{\max(\deg(f),c)}$.

iii: Let $c$ be as in the property. By definition we have $d_{\mathcal{F}} \geq c$ and furthermore we have
\begin{eqnarray*}
V_{d_{\mathcal{F}}} \cap R_{\leq d_{\mathcal{F}}-1} =I \cap R_{\leq d_{\mathcal{F}}-1} \neq V_{d_{\mathcal{F}}-1}.
\end{eqnarray*}

iv: Follows directly from the definitions (Proposition \ref{4213}iii).

v: Follows from Proposition \ref{4213}iv.

vi, vii: Follows directly from the definitions.

viii: Follows since $V_{\mathcal{F},i}=V_{\mathcal{F} \cup \{g\},i}$ if $i \geq j$.

\end{proof}

Note that property iv gives a nice interpretation of the last fall degree: it is the largest degree fall we need to completely get the ideal, hence the name (another name might be \emph{maximal gap degree}, which is more in the spirit of the definition itself). In the next section, we show how one can solve a system once one knows the last fall degree. In heuristics, one often uses the notion of \emph{first fall degree}, the first $c$ such that $V_c \cap R_{\leq c-1} \neq V_{c-1}$ to bound the complexity of Gr\"obner basis algorithms. Actually, most articles, such as \cite{PET}, use a slightly different definition of the first fall degree. They say that the first fall degree $d_{\mathcal{F},f}$ is the first $d \geq \deg(\mathcal{F})$ such that there exists $g_f \in R$ for $f \in \mathcal{F}$ such that $d=\max_{f \in \mathcal{F}}(\deg(g_f f))$ and $\deg( \sum_{f \in \mathcal{F}} g_f f )<d$ and $\sum_{f \in \mathcal{F}} g_f f \neq 0$. By definition we have $d_{\mathcal{F},f} \leq d_{\mathcal{F}}$ if $d_{\mathcal{F}} \geq \deg(\mathcal{F})$ and $d_{\mathcal{F}}>0$. We do not think that the first fall degree is the right notion for the complexity of such algorithms (see also \cite{KO14}). We will derive complexity bounds for solving systems based on the last fall degree.

Property ii in combination with iii gives a method (using a monomial order and a Gr\"obner basis computation) to compute the last fall degree. It would be of great importance to find a method which does not use a monomial order.

\begin{remark}
Let $\mathcal{F}$ be  a finite subset of $R$. It is in general not true that $V_{\mathcal{F},d_{\mathcal{F}}}$ generates the same ideal as $\mathcal{F}$. For example, if $m=1$ and $\mathcal{F}=\{f\}$ with $f$ not constant, then one has $d_{\mathcal{F}}=0$, whereas $V_{\mathcal{F},0}$ does not generate $(f)$.  
\end{remark}

\subsection{Solving systems}

We will now discuss how one can solve a multivariate zero-dimensional system once the last fall degree is known.

\begin{proposition} \label{666}
Let $k$ be a field. Assume that one can factor a polynomial of degree at most $t$ using a number of field equations which is polynomial in $g(t)$ where $g$ is some function. 
Let  $\mathcal{F} \subset R$ be a finite set.  Assume that the ideal $I$ generated by $\mathcal{F}$ is radical and that the system has at most $e$ solutions over $\overline{k}$. Set $d=\max(d_{\mathcal{F}},e)$. Then one can find all solutions of $I$ in $k$ in a number of field operations which is polynomial in the cardinality of $\mathcal{F}$, $g(d)$ and $(m+d)^d$.
\end{proposition}

\begin{proof}
Compute $V_d$ with a number of field operations polynomial in the input size of $\mathcal{F}$ and $(m+d)^d$ (Proposition \ref{4213}i). We will work in $V_d$ to find all the solutions.

Assume that all solutions over $\overline{k}$ of the system are
\begin{eqnarray*}
Z(\mathcal{F}) = \{(a_{0,0},\ldots,a_{0,m-1}), \ldots, (a_{t,0},\ldots,a_{t,m-1}) \} \subset \overline{k}^m
\end{eqnarray*}
with $t < e$.
Since $I$ is a radical ideal, by the Nullstellensatz and Galois theory, one has
\begin{eqnarray*}
h_0 = \prod_{a \in \{a_{i,0}: i=0,\ldots,t\}} (X_0-a) \in I.
\end{eqnarray*}
Using linear algebra, and the definition of the last fall degree, one can find $h_0$ as the nonzero polynomial of minimal degree $d_0$ in $V_d \cap \mathrm{Span}_{k}\{1,X_0,\ldots,X_0^e\}$.
Factor $h_0$ with a number of operations polynomial in $g(t)$. Assume that $a_0$ is a root of $h_0$ in $k$. We will find all solutions over $k$ with $X_0=a_0$. Set $h_0'=h_0/(X_0-a_0)$ of degree $d_0-1$.  By the Nullstellensatz and Galois theory, one has
\begin{eqnarray*}
h_1= h_0' \prod_{a \in \{a_{i,1}: i=0,\ldots,t, a_{i,0}=a_0 \}}  (X_1-a) \in I.
\end{eqnarray*}
Using linear algebra, one finds $h_1$ as the polynomial of minimal degree $d_1$ in $V_d \cap \mathrm{Span}_{k}\{h_0',X_1 h_0',\ldots,X_1^{e-d_0+1} h_0' \}$.
Factor $h_1/h_0'$ over $k$. Pick a solution $a_1$ over $k$ and find all solutions with $X_0=a_0$, $X_1=a_1$ using the similar recursive procedure. Hence one can find all solutions over $k$ with the claimed number of field operations.
\end{proof}

If $k$ is a finite field of cardinality $q$, one can factor a polynomial of degree bounded by $t$ with operations polynomial in  $\max(\log(q),t)$ in a probabilistic way and $\max(q,t)$ in a deterministic way \cite{GAT}.

\subsection{Comparison} \label{906}

In this subsection we will compare the above approach of solving a system $\mathcal{F}$ with other methods. 

The construction of the $V_i$ above is quite similar to operations done using algorithms like XL (see for example \cite{COU}), although we `use' relations which cause the degree to fall (see for example MutantXL, \cite{BUC}). Our method for solving the system itself (Proposition \ref{666}) is more in the spirit of using a lexicographic order to solve the system. 

Given a system $\mathcal{F}$, in practice, one often does not know $d_{\mathcal{F}}$. One can then solve the system by increasing $i$ and computing the $V_i$ until one has the right projection polynomials as in the proof of Proposition \ref{666}. This is the main idea of MutantXL (see \cite{BUC}).

From a complexity point of view, the last fall degree also shows that under certain circumstances MutantXL (or the above described method) is faster than the standard Gr\"obner basis methods. Indeed, suppose that the system $\mathcal{F}$ has $s \leq d_{\mathcal{F}}$ solutions. Then one can solve the system by looking at $V_{d_{\mathcal{F}}}$ (Proposition \ref{666}). Note that $d_{\mathcal{F}}$ is not more than the degree needed to compute a Gr\"obner basis for any monomial order (Proposition \ref{4444}ii). Hence the new algorithm might terminate at a lower degree than a Gr\"obner basis algorithm. If this happens, this usually means that the MutantXL approach is faster.

From a practical point of view, we did not really address how to construct the $V_i$ as efficiently as possible. To construct these $V_i$ in an efficient way, one has to try to keep matrices as sparse as possible and do as few as possible reductions. Algorithms such as $F_4$, $F_5$ \cite{FAU4, FAU5} or MutantXL \cite{BUC} should help to achieve this.

We hope that the framework with the last fall degree allows one to prove complexity statements of solving certain systems. Our framework has the advantage that it behaves well with respect to various operations (Proposition \ref{4444}) and that we do not use a monomial order.
For example, it allows us to compare the last fall degree of a system before and after Weil descent, without using heuristic assumptions (Theorem \ref{1234}).

\section{Weil descent} \label{302}

Let $q$ be a prime power. Let $n \in \mathbb{Z}_{\geq 1}$ and let $k$ be a finite field of cardinality $q^n$. Let $k'$ be the subfield of $k$ of cardinality $q$. In this section, we introduce two Weil descent transforms for a finite subset of $R=k[X_0,\ldots,X_{m-1}]$.

Let $\mathcal{F} \subset R$ be a finite set of polynomials. Suppose we want to find the common zeros of these polynomials in $k$.  Let $I$ be the ideal generated by 
\begin{eqnarray*}
\mathcal{F}_f=\mathcal{F} \cup \{ X_i^{q^n}-X_i: i=0,\ldots,m-1\}.
\end{eqnarray*}

\subsection{Weil descent} \label{551}

Let $\alpha_0,\ldots,\alpha_{n-1}$ be a basis of $k/k'$. Write $X_i= \sum_{j=0}^{n-1} \alpha_j X_{ij}$. For $f \in \mathcal{F}$ and $j=0,\ldots,n-1$, we define $[f]_{j} \in k'[X_{ij}, i=0,\ldots,m-1, j=0,\ldots,n-1]$ by
\begin{eqnarray*}
f(\sum_{j=0}^{n-1} \alpha_j X_{0j}, \ldots, \sum_{j=0}^{n-1} \alpha_j X_{m-1\ j}) \equiv \sum_{j=0}^{n-1} [f]_{j} \alpha_j && \pmod{X_{ij}^q-X_{ij},\ i=0,\ldots, \\
&&\ \ \ \  \ m-1, j=0,\ldots,n-1}
\end{eqnarray*}
where $[f]_j$ is chosen of minimal degree (so $\deg_{X_{ij}}([f]_k) \leq q-1$).
The system 
\begin{eqnarray*}
\mathcal{F}'=\{[f]_j:\ f \in \mathcal{F}, j=0,\ldots,n-1\}
\end{eqnarray*}
is called the \emph{Weil descent system} of $\mathcal{F}$ with respect to $\alpha_0,\ldots,\alpha_{n-1}$. 
There is a bijection between the solutions over $k$ (or $\overline{k}$) of $\mathcal{F}_f$ and the solutions over $k'$ (or $\overline{k}$) of 
\begin{eqnarray*}
\mathcal{F}'_f=\mathcal{F}' \cup \{ X_{ij}^q-X_{ij}: i=0,\ldots,m-1,\ j=0,\ldots,n-1\}. 
\end{eqnarray*}
Note that the ideals generated by $\mathcal{F}_f$ and $\mathcal{F}'_f$ are radical ideals.

An interesting choice for the $\alpha_i$ is a normal basis, that is, a basis with $\alpha_i=\theta^{q^i}$ for some $\theta \in k$. Such a basis always exists.

\begin{remark} \label{812}
A different choice of $\alpha_i$ merely results in a linear change of the variables $X_{ij}$ and a linear change of the polynomials $[f]_i$ and the field equations $X_{ij}^q-X_{ij}$.
Indeed, if $\beta_0,\ldots,\beta_{n-1}$ is another basis, then we can write $\beta_i=\sum_{j=0}^{n-1} c_{ij} \alpha_j$ and $\alpha_i = \sum_{j=0}^{n-1} d_{ij} \beta_j$. Let $C=(c_{ij})_{i,j}$ be the corresponding matrix. One has:
\begin{eqnarray*}
f(\sum_{j=0}^{n-1} \beta_j X_{0j}, \ldots, \sum_{j=0}^{n-1} \beta_j X_{m-1\ j}) &=&  f( \sum_{k=0}^{n-1}  \alpha_k  \sum_{j=0}^{n-1} c_{jk} X_{0j}, \ldots, \sum_{k=0}^{n-1}  \alpha_k  \sum_{j=0}^{n-1} c_{jk} X_{m-1\ j}  ) \\
&\equiv& \sum_{i=0}^{n-1}\mathrm{diag}(C,\ldots,C)  [f]_i \alpha_i \\
&=&\sum_{j=0}^{n-1} \left(   \sum_{i=0}^{n-1}  d_{ij} \mathrm{diag}(C,\ldots,C)  [f]_i \right) \beta_j. 
\end{eqnarray*}
If $d$ is the last fall degree of $\mathcal{F}_f'$ with respect to the $\alpha_i$, and $d'$ with respect to the $\beta_i$, we conclude that $\deg(\mathcal{F}')$ does not depend on the choice of basis and that
\begin{eqnarray*}
\max(d,\deg(\mathcal{F}'),q)=\max(d',\deg(\mathcal{F}'),q).
\end{eqnarray*}
\end{remark} 

\subsection{Another model for Weil descent}

For practical reasons, we will often work with another model of Weil descent.

Let $S=k[X_{ij}: i=0,\ldots,m-1,\ j=0,\ldots, n-1]$. Let $e_0,\ldots,e_{m-1} \in \Z_{\geq 0}$. Let $X_{i}^{e_i'}$ be the remainder of division of $X_i^{e_i}$ by $X_i^{q^n}-X_i$. Write $e_i'= \sum_{j=0}^{n-1} e_{ij}'q^j$ in base $q$ with $e_{ij}' \in \{0,1,\ldots,q-1\}$. 
We set
\begin{eqnarray*}
\overline{\prod_{i=0}^{m-1} X_i^{e_i}} = \prod_{i=0}^{m-1} X_{i0}^{e'_{i0}} \cdots X_{i\ n-1}^{e_{i\ n-1}'} \in S.
\end{eqnarray*}
We extend this definition $k$-linearly for all polynomials in $R$. This gives a map $\bar{}: R \to S$.
We set
\begin{eqnarray*}
\overline{\mathcal{F}}=\{ \overline{f}: f \in \mathcal{F} \}
\end{eqnarray*}
and we set, where by convention $X_{in}=X_{i0}$,
\begin{eqnarray*}
\overline{\mathcal{F}}_f = \overline{\mathcal{F}} \cup \{X_{ij}^q-X_{i\ j+1}:\ i=0,\ldots,m-1,\ j=0,\ldots,n-1 \}.
\end{eqnarray*}
We let $\overline{I}$ be the ideal generated by $\overline{\mathcal{F}}_f$. Note that $\overline{I}$ is radical.

There is a bijection between the zero set of $I$ (over $k$ or $\overline{k}$) and that of $\overline{I}$ (over $k$ or $\overline{k}$). If for example $X_i=a_i \in \overline{k}$ gives a zero of $I$, then $(X_{i0},\ldots,X_{i\ n-1})=(a_i,a_i^q,\ldots,a_i^{q^{n-1}})$ gives a zero of $\overline{I}$. 

We will now prove a couple of lemmas which will be useful later.

\begin{lemma} \label{544}
Let $h_1, h_2 \in R$, $g \in S$. One has, where $\equiv_i$ is defined with respect to $\overline{\mathcal{F}}_f$:
\begin{enumerate}
\item $ \overline{h_1+h_2} \equiv_{\max(\deg(\overline{h_1}),\deg(\overline{h_2}))} \overline{h_1}+\overline{h_2}$;
\item $\overline{h_1} \cdot \overline{h_2} \equiv_{\deg(\overline{h_1})+\deg(\overline{h_2})} \overline{h_1 h_2}$;
\item There is $h_3 \in R$ with $\deg_{X_i}(h_3)<q^n$ such that $g \equiv_{\deg(g)} \overline{h_3}$.
\end{enumerate}
\end{lemma}
\begin{proof}
One reduces to the case of monomials and the result then follows easily.
\end{proof}

We have a morphism of $k$-algebras $\varphi: S \to R$ which maps $X_{ij}$ to $X_i^{q^j}$. This map has the following properties.

\begin{lemma} \label{943}
Let $h \in R$. The following statements hold:
\begin{enumerate}
\item $\varphi(\overline{h}) \equiv h \pmod{X_i^{q^n}-X_i,\ i=0,\ldots,m-1}$;
\item $h \in I$ if and only if $\overline{h} \in \overline{I}$.
\end{enumerate}
\end{lemma}
\begin{proof}
i: Follows directly.

ii: Let $h \in I$. We will show $\overline{h} \in \overline{I}$. One can write $h= \sum_{i=0}^{m-1} b_i (X_i^{q^n}-X_i) + \sum_{f \in \mathcal{F}} a_f f$. Modulo $\overline{I}$ we find with Lemma \ref{544}:
\begin{eqnarray*}
\overline{h} =\overline{  \sum_{i=0}^{m-1} b_i (X_i^{q^n}-X_i) + \sum_{f \in \mathcal{F}} a_f f} \equiv \sum_{i=0}^{m-1} \overline{b_i} (X_{i0}-X_{i0})+\sum_{f \in \mathcal{F}}^r \overline{a_f} \overline{f} \equiv 0.
\end{eqnarray*}

Conversely, let $h \in R$ and assume $\overline{h} \in \overline{I}$. 
Write $\overline{h}=  \sum_{i=0}^{m-1} \sum_{j=0}^{n-1} c_{ij} (X_{i j}^q-X_{i\ j+1})+ \sum_{f \in \mathcal{F}} b_f\overline{f}$. One finds, using i,
\begin{eqnarray*}
\varphi(\overline{h}) &=&   \sum_{i=0}^{m-1} \sum_{j=0}^{n-1} \varphi(c_{ij}) \varphi(X_{ij}^q-X_{i\ j+1})+  \sum_{f \in \mathcal{F}} \varphi(b_f) \varphi(\overline{f}) \\
&\equiv& \sum_{i=0}^{m-1} \varphi(c_{i\ n-1}) (X_i^{q^n}-X_i) +  \sum_{f \in \mathcal{F}} \varphi(b_f) f  \pmod{X_i^{q^n}-X_i,\ i=0,\ldots,n-1}.
\end{eqnarray*}
We conclude $\varphi(\overline{h}) \in I$. 
\end{proof}

\subsubsection{Degree bounds}
\begin{lemma} \label{5552}
Let $g \in R \setminus k$. Then one has
\begin{eqnarray*}
\deg( \overline{g} ) \leq \lfloor m(q-1) \left( \log_q(\frac{\deg(g)}{m})+1 \right) \rfloor.
\end{eqnarray*}
\end{lemma}
\begin{proof}
Let $g \in k[X] \setminus k$. Then one has
\begin{eqnarray*}
\deg( \overline{g} ) \leq (q-1) \left( \log_q( \deg(g)  )+ 1 \right).
\end{eqnarray*}
Let $g \in R \setminus k$. It is enough to prove the result for monomials. Assume that $g=X_0^{a_0}\cdots X_{m-1}^{a_{m-1}}$. Then by the first part and the inequality of arithmetic and geometric means, one has
\begin{eqnarray*}
\deg(\overline{g}) &\leq& \sum_{i=0}^{m-1} (q-1) \left( \log_q( a_i  )+ 1 \right) = (q-1) \left( \log_q( \prod_{i=0}^{m-1} a_i) +m \right) \\
&\leq&  (q-1) \left( \log_q( \left( \frac{1}{m} \sum_{i=0}^{m-1} a_i \right)^m) +m \right) = m(q-1) \left( \log_q(\frac{\deg(g)}{m})+1 \right). 
\end{eqnarray*}
\end{proof}

\begin{lemma} \label{5551}
Let $i \in \Z_{\geq 0}$. Set $s=\tau(i,q,m)$. 
Then one has
\begin{eqnarray*}
\overline{V_{\mathcal{F}_f,i}} \subseteq V_{\overline{\mathcal{F}}_f,s}.
\end{eqnarray*}
\end{lemma}
\begin{proof}
Assume $i>0$. Let $f \in \mathcal{F}$ non constant with $\deg(f)\leq i$. Then Lemma \ref{5552} gives $\overline{f} \in V_{\overline{\mathcal{F}}_f,s}$.
Assume $g \in V_{\mathcal{F}_f,i}$, $h \in R$ both non constant such that $\deg(gh) \leq i$. Note that $\overline{gh} \equiv_{\overline{\mathcal{F}}_f, \deg(\overline{g})+\deg(\overline{h})} \overline{g} \overline{h}$ by Lemma \ref{544}ii. Then Lemma \ref{5552} gives, together with the the inequality of arithmetic and geometric means,
\begin{eqnarray*}
\deg(\overline{g}\overline{h})=\deg(\overline{g})+\deg(\overline{h}) &\leq& m(q-1)  \left( \log_q(  \frac{\deg(g)}{m})+ 1 \right) \\
&& + m (q-1)  \left( \log_q(  \frac{\deg(h)}{m})+ 1 \right)\\
&\leq& 2 m(q-1) \left(\log( \frac{i}{2m}  )+1 \right).
\end{eqnarray*}
The result then follows easily.
\end{proof}

\section{Last fall degree and descent} \label{303}

\subsection{Relating the types of Weil descent}

Let $k$ be a finite field of cardinality $q^n$ and let $k'$ be the subfield of $k$ of cardinality $q$. Let $\mathcal{F} \subset R$ be a finite subset. We will now compare the systems $\overline{\mathcal{F}}_f$ and $\mathcal{F}'_f$ with respect to a normal basis $\{\theta, \theta^q,\ldots,\theta^{q^{n-1}}\}$ of $k/k'$. We imitate a proof from Granboulan et al. \cite[Section 4.2]{GRA}.

\begin{proposition} \label{yang}
One has:
\begin{eqnarray*}
\max(d_{\mathcal{F}'_f},q,\deg(\mathcal{F}')) \leq \max(d_{\overline{\mathcal{F}}_f},q,\deg(\mathcal{F}'))
\end{eqnarray*}
\end{proposition}
\begin{proof}
Set
\begin{eqnarray*}
\mathcal{G}=\{\overline{f}, \overline{f^q} ,...,\overline{f^{q^{n-1}}} :f\in\mathcal{F}\} \cup \{X_{ij}^q-X_{i\ j+1}: i=0,\ldots,m-1,\  j=0,\ldots,n-1\}.
\end{eqnarray*}
Note that we have $\overline{\mathcal{F}}_f \subseteq \mathcal{G}$. Note furthermore that both sets generate the same ideal since
\begin{eqnarray*} 
\overline{f^{q^l}} \equiv_{\overline{\mathcal{F}}_f,\infty} \overline{f}^{q^l}
\end{eqnarray*}
by Lemma \ref{544}ii. Hence we have $d_{\mathcal{G}} \leq d_{\overline{\mathcal{F}}_f}$ (Proposition \ref{4444}vi, vii).

Since $k/k'$ is a separable extension, the matrix $(\theta^{q^{i+j}})_{i,j=0}^{n-1}$ is invertible. Consider the linear change of variables defined by
\begin{eqnarray*}
Y_{ij}=\sum_{k=0}^{n-1}\theta^{q^{j+k}}X_{ik}.
\end{eqnarray*}
By convention, we set $Y_{ij}=Y_{i\ j \pmod{n}}$. 
We first notice that the field equations of the two systems are the same up to a linear change of equations:
\begin{eqnarray*}
 Y_{ij}^q-Y_{i\ j+1}&=& \sum_{k=0}^{n-1}\theta^{q^{j+k+1}}X_{ik}^q + \sum_{k'=0}^{n-1}\theta^{q^{j+1+k'}}X_{ik'}\\
&=&  \sum_{k=0}^{n-1}\theta^{q^{j+k+1}}(X_{ik}^q-X_{ik}).
\end{eqnarray*}

We claim:
\begin{eqnarray*}
 \overline{f^{q^l}}(\ldots,Y_{ij},\ldots) 
\equiv\sum_{k=0}^{n-1}\theta^{q^{k+l}}[f]_k\pmod{X_{ij}^q-X_{ij}, i=0,\ldots,m-1,\ j=0,\ldots,n-1}.
\end{eqnarray*}
It is enough to prove the claim for $f=c\prod_{i=0}^{m-1}X_i^{e_i}$, since both Weil descent models are additive.

Let $X_i^{e_i'}$ be the remainder of division of $X_i^{e_i}$ by ${X_i^{q^n}-X_i}$ and $e_i'=\sum_{j=0}^{n-1} a_{ij}q^j$ with $a_{ij} \in \{0,1,\ldots,q-1\}$.

This gives modulo $Y_{ij}^q-Y_{i\ j+1}$
\begin{eqnarray*}
\overline{f^{q^l}}(...,Y_{ij},...)=c^{q^l}\prod_{i=0}^{m-1}\prod_{j=0}^{n-1}Y_{i\ j+l}^{a_{ij}}.
\end{eqnarray*}
Furthermore, modulo $X_{ij}^q-X_{ij}$, we have
\begin{eqnarray*}
f^{q^l}(\ldots,\sum_{k=0}^{n-1}\theta^{q^k}X_{ik},\ldots) &=& c^{q^l}\prod_{i=0}^{m-1}(\sum_{k=0}^{n-1}\theta^{q^k}X_{ik})^{q^l\sum_{j=0}^{n-1} a_{ij}q^j} \\
&\equiv& c^{q^l}\prod_{i=0}^{m-1}\prod_{j=0}^{n-1}(\sum_{k=0}^{n-1}\theta^{q^{k+l+j}}X_{ik})^{a_{ij}}.
\end{eqnarray*}
Thus we get the following equation from the above two identities modulo $X_{ij}^q-X_{ij}$, since $[f]_k^q\equiv[f]_k$:
\begin{eqnarray*}
 \overline{f^{q^l}}(\ldots,Y_{ij},\ldots)\equiv f^{q^l}(\ldots,\sum_{k=0}^{n-1}\theta^{q^k}X_{ik},\ldots) \equiv \left(\sum_{k=0}^{n-1}\theta^{q^{k}}[f]_k \right)^{q^l} 
\equiv \sum_{k=0}^{n-1}\theta^{q^{k+l}}[f]_k.
\end{eqnarray*}
In other words, there exist polynomials $h_{ij}^{(l)}\in S$, such that
\begin{eqnarray*}
\overline{f^{q^l}}(\ldots,Y_{ij},\ldots)=\sum_{k=0}^{n-1}\theta^{q^{k+l}}[f]_k+\sum_{i,j}h_{ij}^{(l)}(X_{ij}^q-X_{ij}).
\end{eqnarray*}
One has $\deg(\overline{f^{q^l}})=\deg(\overline{f})=\max_k(\deg([f]_k))$ by \cite[Proposition 3.2]{KO13}. Since $\{X_{ij}^q-X_{ij}: i=0,\ldots,m-1, j=0,\ldots,n-1\}$ forms a Gr\"obner basis for any graded order, we conclude that $\deg(h_{ij}^{(l)}(X_{ij}^q-X_{ij})) \leq \deg(\overline{f^{q^l}})$. 

Hence we have shown that the systems $\mathcal{G}$ and $\mathcal{F}_f'$ can be obtained from each other through a linear change of variables and a change of polynomials. From Proposition \ref{4444}iv,v we conclude
\begin{eqnarray*}
\max(d_{\mathcal{F}'_f},q,\deg(\mathcal{F}'))=\max(d_{\mathcal{G}},q,\deg(\mathcal{F}'))  \leq \max(d_{\overline{\mathcal{F}}_f},q,\deg(\mathcal{F}')).
\end{eqnarray*}
\end{proof}

\subsection{GCD computations}

Let $q$ be a prime power and let $k$ be a finite field of cardinality $q^n$. Let $\mathcal{F} \subset k[X]$ be a finite set. Consider the Weil descent system $\overline{\mathcal{F}}_f$ to the subfield of cardinality $q$. Define $\equiv_j$ with respect to $\overline{\mathcal{F}}_f$.
For $e \in \Z_{\geq 0}$ with $e= \sum_{i} a_i q^i$ in base $q$, we set $w(e)=\sum_i a_i$. For $f=\sum_i b_i X^i$, we set $w(f) = \max(w(i): b_i \neq 0)$. Note that $w(f) \geq \deg(\overline{f})$, with equality if $\deg(f)<q^n$.

We start with a technical lemma.

\begin{lemma} \label{540}
 Let $h_2 \in k[X]$ nonzero of degree $d$. Set $u=\tau(2d,q,1)$. Assume $\overline{h_2} \equiv_{u} 0$.  Let $h_1 \in k[X]$.
Let $h_3$ be the remainder of division of $h_1$ by $h_2$. Then one has $\overline{h_1}  \equiv_{\max(u,w(h_1))} \overline{h_3}$.
\end{lemma}
\begin{proof}
If $d=0$, the result follows easily. Assume $d>0$.

Fix $h_2$ and write $h_2=\sum_{i=0}^{d} b_i X^i$ where $b_{d} \neq 0$.
Since taking remainders is additive, it suffices to prove the result for $h_1=X^e$. Let $r_e$ be the remainder of division of $X^e$ by $h_2$. For $g \in k[X]$ with $\deg(g) \leq d$, one has $\deg(\overline{g}) \leq u/2$ (Lemma \ref{5552}). In particular, we have $\deg(\overline{r_e}) \leq u/2$.

We will prove the following statements successively:
\begin{enumerate}
\item for $e \in \{0,1,\ldots,qd-1\}$, we have $\overline{X^e} \equiv_{u} \overline{r_e}$;
\item if $e, e'$ satisfy $w(e)+w(e') \leq u$, $\overline{X^e} \equiv_{u} \overline{r_e}$ and $\overline{X^{e'}} \equiv_{u} \overline{r_{e'}}$, then
$\overline{X^{e+e'}} \equiv_{u} \overline{r_{e+e'}}$;
\item for $e$ with $w(e) \leq u$, we have $\overline{X^e} \equiv_{u} \overline{r_e}$;
\item one has $\overline{X^e} \equiv_{\max(u,w(e))} \overline{r_e}$.
\end{enumerate}

i: For $e=0,\ldots,d-1$, the remainder is $X^e$ itself and the result follows. One has $r_{d}= \frac{-1}{b_{d}} \sum_{i=0}^{d-1} b_i X^i$ and this gives $\overline{X^{d}} \equiv_{u} \overline{r_{d}}$. We continue by induction. Assume the statement holds for cases smaller than $e$ and that $e \leq qd-1$. We will prove the statement for $e$. Write $r_{e-1}=\sum_{j=0}^{d-1} c_j X^j$. Note that $r_e$ is the remainder of division of $X r_{e-1}$ by $h_2$, which gives $r_e=\sum_{j=0}^{d-1} c_j r_{j+1}$. Note that $e-1\leq qd-2=q^{\log_q(d) +1}-2$. Hence we have (as $d>0$, see also Lemma \ref{5552}):
\begin{eqnarray*}
\deg(\overline{X}) + \deg(\overline{X^{e-1}})  \leq 1 + \lfloor (q-1) \left(\log_q(d) + 2  \right) -1 \rfloor=   \lfloor (q-1) \left( \log_q(d) + 2  \right)\rfloor \leq u.
\end{eqnarray*}

Using Lemma \ref{544} and the induction hypothesis, we find
\begin{eqnarray*}
\overline{X^e} \equiv_{u} \overline{X} \cdot \overline{X^{e-1}} \equiv_{u} \overline{X} \cdot \overline{r_{e-1}} \equiv_{u}
\overline{\sum_{j=0}^{d-1} c_j X^{j+1}}
 \equiv_{u} \overline{ \sum_{j=0}^{d-1} c_j r_{j+1}},
\end{eqnarray*}
and this gives the required remainder.

ii: Assume without loss of generality that $w(e') \leq u/2$. Then one has $u \geq \max(w(e)+w(e'), \deg(\overline{r_e})+w(e'), \deg(\overline{r_e})+\deg(\overline{r_{e'}}))$ and one has $\deg(r_e r_{e'})\leq 2d-2 \leq qd-1$.
Lemma \ref{544} and i give
\begin{eqnarray*}
\overline{X^{e+e'}} \equiv_{u} \overline{X^e} \cdot \overline{X^{e'}} \equiv_{u} \overline{r_e} \cdot \overline{X^{e'}}  \equiv_{u} \overline{r_e} \cdot \overline{r_{e'}} \equiv_{u} \overline{r_e r_{e'}} \equiv_u \overline{r_{e+e'}}.
\end{eqnarray*}

iii: Using ii and induction, we easily reduce to the case where $e=q^i$. Note that $q^i= q \cdot q^{i-1}$ and that $u \geq q$. We can then apply ii and the proof follows by induction.

iv: We prove this statement by induction on $w(e)>u$. Write $e=e_1+e_2$ with $u \leq w(e_1)<w(e)$, and $w(e_1)+w(e_2)=w(e)$. One has (Lemma \ref{544} and iii)
\begin{eqnarray*}
\overline{X^e} &\equiv_{\max(u,w(e))}& \overline{X^{e_1}} \cdot \overline{X^{e_2}} \equiv_{\max(u,w(e))} \overline{r_{e_1}} \cdot \overline{X^{e_2}} \\
&\equiv_{\max(u,w(e))}& \overline{r_{e_1}} \cdot \overline{r_{e_2}} \equiv_{\max(u,w(e))} \overline{r_{e}}.
\end{eqnarray*}
\end{proof}

\begin{proposition}  \label{901}
Assume $\mathcal{F}=\{f\}$ with $f$ nonzero. Set $u=\tau(2 \deg(f), q,1)$ and set $g=\gcd(f,X^{q^n}-X)$.
We have: $\overline{g} \in V_u$.
\end{proposition}
\begin{proof}
Let $f_1$ be the remainder of division of $X^{q^n}-X$ by $f$. By Lemma \ref{540}, we have $\overline{f_1} \equiv_u 0$. Let $f_2$ be the remainder of division of $f$ by $f_1$. Similarly, we find $\overline{f_2} \equiv_u 0$. Hence we can follow the Euclidean algorithm and we obtain $\overline{g} \in V_u$.
\end{proof}

\subsection{Last fall degree of Weil descent systems}

For a finite subset $\mathcal{F} \subset R$, we denote by $Z(\mathcal{F})$ the set of zeros of $\mathcal{F}$ over $\overline{k}$. Let $k''$ be a field extension of $k$. For $i=0,\ldots,m-1$, we write
\begin{eqnarray*}
\pi_{i,\mathcal{F},k''} = \prod_{x \in \{x_i:\ \exists (x_0,\ldots,x_{m-1}) \in Z(\mathcal{F}) \cap k''^m \}} (X_i-x) \in k[X_i].
\end{eqnarray*} 
We write $\pi_{i,\mathcal{F}}$ for $\pi_{i,\mathcal{F},\overline{k}}$. 

We are finally ready to prove the main theorem (Theorem \ref{1234}).

\begin{theorem} \label{1235}
Let $k$ be a finite field of cardinality $q^n$. Let $\mathcal{F} \subset R$ be a finite subset. Let $I$ be the ideal generated by $\mathcal{F}$. Assume that the following hold:
\begin{itemize}
\item $I$ is zero-dimensional, say one has $|Z(\mathcal{F})| \leq s$;
\item $I$ is radical;
\item there is a coordinate $t$ such that the projection map $Z(\mathcal{F}) \to \overline{k}$ to coordinate $t$ is injective;
\end{itemize}
Let $\mathcal{F}_f'$ be the Weil descent system of $\mathcal{F}$ to the subfield $k'$ of cardinality $q$ using some basis of $k/k'$, together with the field equations (Subsection \ref{551}). Then one has
\begin{eqnarray*}
d_{\mathcal{F}'_f} \leq \max\left(\tau(\max(d_{\mathcal{F}}, \deg(\mathcal{F}), (m+1)s,1),q,m),m \cdot \tau(2s,q,1), q \right).
\end{eqnarray*}
\end{theorem}
\begin{proof}
We have $d_{\mathcal{F}'_f} \leq \max(d_{\overline{\mathcal{F}}_f},q, \tau(\deg(\mathcal{F}),q,m))$ by Proposition \ref{yang},  Lemma \ref{5551}, Remark \ref{812} and Proposition \ref{4444}iv, v. Hence we will work with the alternative Weil descent system $\overline{\mathcal{F}}_f$.

Without loss of generality, we may assume that $t=0$. We can then write
\begin{eqnarray*}
Z(\mathcal{F}) = \{(a, \gamma_1(a), \ldots,\gamma_{m-1}(a)): a \in \overline{k}, \pi_{0,\mathcal{F}}(a)=0\}
\end{eqnarray*}
for some $\gamma_i \in k[X_0]$ of degree $< s$ by the Lagrange interpolation formula and by Galois theory. Indeed, we can just put
\begin{eqnarray*}
\gamma_i = \sum_{x=(x_0,\ldots,x_{m-1}) \in Z(\mathcal{F})} x_i \prod_{(x_0',\ldots,x_{m-1}') \in Z(\mathcal{F}) \setminus \{x\}} \frac{X_0-x_0'}{x_0-x_0'}.
\end{eqnarray*}
Note that $\gcd(\pi_{0,\mathcal{F}}, X_0^{q^n}-X_0)=\pi_{0,\mathcal{F},k}$ and one also has
\begin{eqnarray*}
Z(\mathcal{F}) \cap k^n= \{(a, \gamma_1(a), \ldots,\gamma_{m-1}(a)): a \in \overline{k}, \pi_{0,\mathcal{F},k}(a)=0 \}.
\end{eqnarray*}

Set $r_0=\max(d_{\mathcal{F}},s,1)$. By definition we have $\pi_{i,\mathcal{F}}, X_j-\gamma_j \in V_{\mathcal{F},r_0}$, since $I$ is radical. Set $r_1=\tau(r_0,q,m)$. By Lemma \ref{5551}, we have $\overline{\pi_{i,\mathcal{F}}}, \overline{X_j-\gamma_j} \in  V_{\overline{\mathcal{F}}, r_1}$.
Set $r_2=\max(r_1,\tau(2s,q,1))$. We have $\overline{\pi_{0,\mathcal{F},k}},\overline{\pi_{j,\mathcal{F}}}, \overline{X_j-\gamma_j} \in V_{\overline{\mathcal{F}}_f, r_2}$ (for $j=1,\ldots,m-1$) by Proposition \ref{901}.

Now consider the system
\begin{eqnarray*}
\mathcal{G} = \{\pi_{0,\mathcal{F},k}, \pi_{1,\mathcal{F}}, \ldots,\pi_{m-1,\mathcal{F}}\} \cup \{X_1-\gamma_1,\ldots,X_{m-1}-\gamma_{m-1}\}.
\end{eqnarray*}
We have $\overline{\mathcal{G}} \subseteq V_{\overline{\mathcal{F}}_f, r_2}$.
Let $I'$ be the ideal generated by $\mathcal{F}_f$. Note that $I'$ is the same as the ideal generated by $\mathcal{G}$, because both ideals are radical and have the same zero set. We first bound $d_{\mathcal{G}}$. Let $h \in I'$. One easily obtains
\begin{eqnarray*}
h \equiv_{\mathcal{G},\deg(h)} h' 
\end{eqnarray*}
for some $h' \in R$ with $\deg_{X_i}(h') <s$ using $\pi_{0,\mathcal{F},k}$ and $\pi_{i,\mathcal{F}}$ ($i=1,\ldots,m-1$). Then one can replace $X_i$ ($i>0$) with $\gamma_i$ and do reductions with $\pi_{0,\mathcal{F},k}$ to make a polynomial in $k[X_0]$ and conclude
\begin{eqnarray*}
h \equiv_{\mathcal{G}, \max(\deg(h),(m+1)s)} 0.
\end{eqnarray*}
Hence we have $d_{\mathcal{G}} \leq (m+1)s$. 

Let $h \in S$.  We first claim that there is $h_1 \in R$ with $\deg_{X_i}(h_1)<s$ and 
\begin{eqnarray*}
h \equiv_{\overline{\mathcal{F}}_f, \max(\deg(h), m \cdot \tau(2s,q,1),r_2)} \overline{h_1}.
\end{eqnarray*}
We may assume that $h$ is a monomial. By Lemma \ref{544}iii, there is a  $h_3 \in R$ with $\deg_{X_i}(h_3)<q^n$ with $h \equiv_{\overline{\mathcal{F}}_f, \deg(h)} \overline{h_3}$. Note that $h_3$ can be chosen to be a monomial, say $h_3=X_0^{a_0}\cdots X_{m-1}^{a_{m-1}}$. Set $w_i=\deg(\overline{X_i^{a_i}})$. Without loss of generality, we may assume $w_0 \geq w_1 \geq \ldots \geq w_{m-1}$. Let $j$ be maximal such that $w_j > \tau(2s,q,1)$. Let $g_i$ be the division of remainder of $X_i^{a_i}$ by $\pi_{i,\mathcal{F}}$ (and by $\pi_{0,\mathcal{F},k}$ if $i=0$). By Lemma \ref{540} for $i=0,\ldots,j$ we have
\begin{eqnarray*}
\overline{X_i^{a_i}} \equiv_{\overline{\mathcal{G}}_f, w_i} \overline{g_i}
\end{eqnarray*}
and for $i=j+1,\ldots,m-1$ we have
\begin{eqnarray*}
\overline{X_i^{a_i}} \equiv_{\overline{\mathcal{G}}_f,\tau(2s,q,1)} \overline{g_i}
\end{eqnarray*}
We find (Remark \ref{yy})
\begin{eqnarray*}
\overline{X_0^{a_0}} \cdots \overline{X_j^{a_j}} \equiv_{\overline{\mathcal{G}}_f, w_0+\ldots+w_j} \overline{g_0} \cdots \overline{g_j}. 
\end{eqnarray*}
We obtain by Lemma \ref{544}ii and Remark \ref{yy}:
\begin{eqnarray*}
h  &\equiv_{\overline{\mathcal{F}}_f, \max(\deg(h),m \cdot \tau(2s,q,1),r_2)}& \overline{X_0^{a_0}} \cdots \overline{X_{m-1}^{a_{m-1}}} \\ 
&\equiv_{\overline{\mathcal{F}}_f, \max(\deg(h),m \cdot \tau(2s,q,1),r_2)}&  \overline{g_0} \cdots \overline{g_j } \cdot \overline{X_{j+1}^{a_{j+1}}} \cdots \overline{X_{m-1}^{a_{m-1}}} \\
&\equiv_{\overline{\mathcal{F}}_f, \max(\deg(h),m \cdot \tau(2s,q,1),r_2)}& \overline{g_0} \cdots \overline{g_{m-1}} \\
&\equiv_{\overline{\mathcal{F}}_f, \max(\deg(h),m \cdot \tau(2s,q,1),r_2)}& \overline{g_0\cdots g_{m-1}}.
\end{eqnarray*}
This finishes the proof of the claim.

Let $\overline{I}$ be the ideal generated by $\overline{\mathcal{F}}_f$. Assume $h \in \overline{I}$. By the above there is $h_1 \in R$ with $\deg_{X_i}(h_1)<s$ and 
\begin{eqnarray*}
h \equiv_{\overline{\mathcal{F}}_f, \max(\deg(h), m \cdot \tau(2s,q,1),r_2)} \overline{h_1}.
\end{eqnarray*}
From Lemma \ref{943} it follows that $h_1 \in I'$. We have $h_1 \in V_{\mathcal{G},(m+1)s}$ by the above. From Lemma \ref{5551} we have $\overline{h_1} \in V_{\overline{\mathcal{G}},\tau((m+1)s,q,m))}$.
Hence we conclude:
\begin{eqnarray*}
h \in V_{\overline{\mathcal{F}}_f, \max(\deg(h), \tau((m+1)s,q,m), m \cdot \tau(2s,q,1), r_2)}
\end{eqnarray*}
where $r_2=\max(r_1,\tau(2s,q,1))=\max(\tau(\max(d_{\mathcal{F}},s,1),q,m),\tau(2s,q,1))$.
Summarizing, this gives
\begin{eqnarray*}
h \in V_{\overline{\mathcal{F}}_f, \max(\deg(h), \tau(\max((m+1)s, d_{\mathcal{F}},1),q,m),m \cdot \tau(2s,q,1)) }.
\end{eqnarray*}
The result then follows.
\end{proof}

\subsection{Possible improvements of the main theorem}

In this subsection, we will discuss how one can improve Theorem \ref{1235}. Our main goal is to obtain a result for which the last fall degree of a Weil descent system does not depend on $n$.

If one reads the proof carefully, one notices that one can replace $(m+1)s$ by $m(s-1)-1+(s-1)=(m+1)(s-1)-1$ if $m>1$. For $m=1$, one can prove a much simpler theorem using mostly Proposition \ref{901}. The result is the following statement.

\begin{theorem} \label{8056}
Let $k$ be a finite field of cardinality $q^n$. Assume $m=1$. Let $\mathcal{F} \subset R$ be a finite subset. Let $d \in \Z_{\geq 0}$ such that there $\exists f \in \mathcal{F}$ with $0 \leq \deg(f) \leq d$, and such that for all $g \in \mathcal{F}$ we have $\deg(\overline{g}) \leq \tau(2d,q,1)$. Let $\mathcal{F}_f'$ be the Weil descent system of $\mathcal{F}$ to the subfield $k'$ of cardinality $q$ using some basis of $k/k'$, together with the field equations (Subsection \ref{551}). Then one has
\begin{eqnarray*}
d_{\mathcal{F}'_f} \leq \max(\tau(2d,q,1), q).
\end{eqnarray*}
\end{theorem}
\begin{proof}
(Sketch) As in the proof of Theorem \ref{1235}, we work with the system $\overline{\mathcal{F}}_f$.

Set $u=\tau(2d,q,1)$ and set $g=\gcd(\mathcal{F} \cup \{X^{q^n}-X\})$. Using Lemma \ref{540} and Proposition \ref{901}, one can prove $\overline{g} \equiv_{u} 0$.

Let $h \in \overline{I}$. By Lemma \ref{544}iii, one has $h \equiv_{\deg(h)} \overline{h_2}$ for some $h_2 \in k[X]$. Since $\overline{h_2} \in \overline{I}$, it follows from Lemma \ref{943}ii that $h_2 \in I$. Hence $h_2$ has remainder $0$ when divided by $g$. From Lemma \ref{540}, we conclude
\begin{eqnarray*}
h \equiv_{\max(\deg(h), u)} \overline{h_2} \equiv_{\max(\deg(h), u)} 0.
\end{eqnarray*}
This finishes the proof.
\end{proof}

One can also study the Weil descent of a system $\mathcal{H}$ which consists of $\mathcal{F}$ and some polynomials in one of the variables of weight at most $\tau(2s,q,1)$ (such as linear subspace constraints). One can easily generalize as in Theorem \ref{8056} and exactly the same result should hold (the extra polynomials do not play a role). We did not use this formulation, because it looks a bit more complex.

The restriction that $I$ is radical, can be removed by using some effective Nullstellensatz.

Consider the condition which says that the projection to one coordinate should be injective. If one has upper bounds on the last fall degree of $\mathcal{F} \cup \{\pi_{i,\mathcal{F},k}: i=0,\ldots,m-1\}$ (this is a system with degree bounded by $\max(\deg(\mathcal{F}),s)$ in $m$ variables), then one can give a similar result without the condition.
Another way to remove this condition on the projection, is the following.
We have the following lemma.
\begin{lemma} \label{446}
Let $k$ be a field, $n \in \Z_{\geq 0}$ and let $v_1,\ldots,v_r \in k^n$ be distinct. Assume that $|k| > {{r}\choose{2}}$. Then there exists a matrix $A \in \mathrm{GL}_n(k)$ such that the first coordinates $Av_1,\ldots,Av_r$ are pairwise distinct.  
\end{lemma}
\begin{proof}
Assume that $k$ is a finite field. Let $q=|k|$. It is equivalent to find $y \in k^n$ such that $\langle y, v_1 \rangle, \ldots, \langle y, v_r \rangle$ are distinct, that is, such that for $i \neq j$ one has $\langle y, v_i-v_j \rangle \neq 0$. There $q^{n-1}$ vectors $y$ with $\langle y, v_i-v_j \rangle =0$. There are at least $q^n-{{r}\choose{2}} q^{n-1}$ vectors which make none of the inner products zero. Hence if $q^n >  {{r}\choose{2}} q^{n-1}$, the result follows. The proof for an infinite field follows in a similar way.
\end{proof}
Hence by enlarging the field $k$, and after applying some transformations, one can make sure the projection maps are injective (use Proposition \ref{4444}). There are some problems when doing this, but an approach along those lines might work.

With our techniques it seems impossible to remove the condition that the system is zero-dimensional (see also Section \ref{230}).

\section{Multi-HFE} \label{304}

In this section we discuss the security of a multi-HFE system. Let us first describe the idea.
The idea of HFE and multi-HFE is that it is easy to solve zero-dimensional systems with few variables, but it becomes harder when the number of variables increases. Using Weil descent, one can construct a system with a lot of variables from a system with only a few variables.

Suppose we have a zero-dimensional system coming from a finite subset $\mathcal{F} \subset R$ where $k$ is a finite field of cardinality $q^n$ with subfield $k'$ of cardinality $q$. If the number of variables is small, then one should be able to find the solutions of the system in $k$ easily with Gr\"obner basis algorithms. Now consider the system $\mathcal{F}'_f$ coming from a Weil descent to $k'$ (in literature, people mostly considered systems which become quadratic after Weil descent (see for example \cite{BET}). Let $\mathcal{G}'$ be the system obtained from a random affine transformation of the variables and a random linear transformation of the polynomials themselves. This system looks very complicated, and it seems hard to find solutions for this system unless one knows the transformations. Theorem \ref{1235}, together with the fact that the last fall degree is almost independent of the linear changes (Proposition \ref{4444}) show that we can give an upper bound on the last fall degree of the Weil descent system $\mathcal{G}'$ which does not depend on $n$. Since we can solve systems if we know the last fall degree (Proposition \ref{666}), we can solve such systems quite efficiently. The dependence on $n$ only comes from Proposition \ref{666}. 

This shows that solving such Weil descent systems is much easier than expected and hence threatens the security of such schemes.

\section{Relation to ECDLP} \label{230}

Let $k$ be a finite field of cardinality $q^n$ and let $k'$ be its subfield of cardinality $q$. Let $f \in R=k[X_0,\ldots,X_{m-1}]$ with $m \geq 2$. It has been suggested (see for example \cite{PET}) that the Weil descent system of $\{f \}$ (or in general a polynomial system which need not be zero-dimensional) to $k'$, the first fall degree is close to the degree of regularity, the largest degree reached during Gr\"{o}bner basis computation.  An example of the Weil descent of a single polynomial  comes from one of the approaches to solve the elliptic curve discrete logarithm problem using summation polynomials (see for example \cite{DIE}). In this case the first fall degree does not depend on $n$ and it is very tempting to adopt the first fall degree assumption as it leads to heuristically subexponential attack on the elliptic curve discrete logarithm problem over finite fields of small characteristics.  However more recent works (see for example \cite{KO14}) have cast serious doubt on the first fall degree assumption.

What we have shown in this paper is that to a large extent the {\em last} fall degree of the Weil descent system of a zero dimensional polynomial system is independent of $n$ (Theorem \ref{1235}).  This has enabled us to successfully solve HFE and multi-HFE systems with rigorously proven time complexity, as the underlying polynomial systems are zero dimensional.  Unfortunately, the system coming from a single multivariate polynomial, without field equations, is not zero-dimensional and our approach using projection polynomials does not work (Theorem \ref{1235}). The system only becomes zero-dimensional when we add the field equations.

We do think that it is of great interest to study such systems coming from a single multivariate polynomial (or systems which are not zero-dimensional). We hope that this article is a step in the right direction.

\end{document}